\documentclass[a4paper,leqno]{article}

\usepackage[T1]{fontenc}
\usepackage[utf8]{inputenc}
\usepackage{mathtools,amssymb,amsthm}
\usepackage{lmodern}
\usepackage{hyperref}
\usepackage{microtype}
\usepackage{color}
\usepackage{enumerate}

\theoremstyle{plain}
\newtheorem{thm}{Theorem}[subsection]
\newtheorem{lem}[thm]{Lemma}
\newtheorem{coro}[thm]{Corollary}
\newtheorem{prop}[thm]{Proposition}

\newtheorem{thmalph}{Theorem}
\newtheorem{coralph}[thmalph]{Corollary}

\theoremstyle{definition}
\newtheorem{rmrk}[thm]{Remark}
\theoremstyle{remark}

\newcommand{\R}{\mathbb{R}}
\newcommand{\Z}{\mathbb{Z}}

\renewcommand{\H}{\mathbb{H}}
\newcommand{\C}{\mathbb{C}}

\newcommand*{\rom}[1]{\romannumeral}

\providecommand{\msc}[1]{{\noindent\small\textbf{Mathematics Subject Classification (2020)} --- #1.}}
\providecommand{\keywords}[1]{{\noindent\small\textbf{Keywords} --- #1.}}

\DeclareMathOperator{\SL}{SL}
\DeclareMathOperator{\PSL}{PSL}
\DeclareMathOperator{\PSO}{PSO}
\DeclareMathOperator{\GL}{GL}

\DeclareMathOperator{\RE}{Re}

\DeclareMathOperator{\tr}{tr}
\DeclareMathOperator{\End}{End}
\DeclareMathOperator{\Vol}{Vol}
\DeclareMathOperator{\rank}{rank}
\DeclareMathOperator{\Id}{Id}
\DeclareMathOperator{\Ad}{Ad}
\DeclareMathOperator{\ad}{ad}

\DeclareMathOperator{\spec}{spec}

\DeclareMathOperator{\Res}{Res}
\DeclareMathOperator{\prim}{prime}

\DeclareMathOperator{\RHS}{RHS}

\numberwithin{equation}{section}

\renewcommand{\thethm}{%
	\ifnum\value{subsection}=0
	\thesection
	\else
	\thesubsection
	\fi
	.\arabic{thm}%
}

\providecommand{\contact}{{
		\bigskip
		\small
		
		\noindent
		\textsc{J.~Frahm}:
		Department of Mathematics, Ny Munkegade 118,  8000, Aarhus C,
		Denmark
		\par\noindent\nopagebreak
		\textit{E-mail address:} \href{mailto:frahm@math.au.dk}
		{\texttt{frahm@math.au.dk}}\\
		
		\noindent
		\textsc{P.~Spilioti}:
		Department of Mathematics, Ny Munkegade 118,  8000, Aarhus C,
		Denmark
		\par\noindent\nopagebreak
		\textit{E-mail address:} \href{mailto:spilioti@math.au.dk}
		{\texttt{spilioti@math.au.dk}}	
}}

\title{Twisted Ruelle zeta function at zero for compact hyperbolic surfaces}
\author{Jan Frahm, Polyxeni Spilioti}

\begin{document}
	
	\maketitle
	
	\begin{abstract}
		Let $X$ be an orientable compact connected hyperbolic surface of genus $g$. In this paper, we prove that the twisted Selberg and Ruelle zeta functions, associated with an arbitrary finite-dimensional complex representation $\chi$ of $\pi_1(X)$ admit a meromorphic continuation to $\C$. Moreover, we study the behaviour of the twisted Ruelle zeta function at $s=0$ and prove that at this point it has a zero of order $\dim(\chi)(2g-2)$.
	\end{abstract}
	
	\bigskip
	\newline
	\msc{Primary: 11M36; Secondary: 11F72, 37C30}

	\section{Introduction}
	
	The zeta functions of Selberg and Ruelle were first introduced in \cite{Rue76,sel}. For a hyperbolic surface, they are defined by an Euler-type product over the prime closed geodesics, i.e., the closed geodesics which are not multiples of shorter geodesics, and in terms of the lengths of these geodesics. 
	Twisting these Euler products by a representation of the fundamental group of the surface defines twisted versions of the Selberg and Ruelle zeta functions. The goal of this paper is to prove that these twisted zeta functions, associated with a non-unitary representation of the fundamental group, admit a meromorphic continuation to $\C$. In addition, we study the behaviour of the twisted Ruelle zeta function near the origin and prove that it is related to a topological invariant, the Euler characteristic of the surface.
	
	Let $\Gamma$ be a discrete, torsion-free, cocompact subgroup of $\PSL_{2}(\R)$ and let $X=\Gamma \backslash\H^{2}$ be the associated compact hyperbolic surface with fundamental group $\pi_1(X)\cong\Gamma$. We define the \textit{twisted} dynamical zeta functions, associated with a finite-dimensional, complex representation $\chi\colon \Gamma\rightarrow \GL(V_{\chi})$
	of $\Gamma$, which is not necessary unitary. For $s\in\C$, the twisted Selberg zeta function, associated with $\chi$, is given by the infinite product
	\begin{equation*}
		Z(s;\chi):=\prod_{\substack{[\gamma]\neq{e}\\ [\gamma]\prim}} \prod_{k=0}^{\infty}\det
		\big(\Id-\chi(\gamma)e^{-(s+k) l(\gamma)}\big),
	\end{equation*}
	and the twisted Ruelle zeta function, associated with $\chi$, is defined by the infinite product
	\begin{equation*}
		R(s;\chi):=\prod_{\substack{[\gamma]\neq{e}\\ [\gamma]\prim}}\det (\Id-\chi(\gamma)e^{-sl(\gamma)}).
	\end{equation*}
	Here, the products run over the \emph{primitive} conjugacy classes of $\Gamma$ (see Subsection~\ref{sec:HyperbolicSurfaces}), which correspond to the prime closed geodesics on the surface, and $l(\gamma)$ denotes the length of geodesic corresponding to $[\gamma]$.
	The products in the definition of the \textit{twisted} zeta functions converge absolutely and uniformly on compact subsets of some right half plane in $\C$ by \cite[Theorem 3.1]{fedosova2020meromorphic} (see also Propositions~\ref{SelbergConvergence} and \ref{prop:RuelleConvergence}). If we consider $\chi$ to be the trivial representation, we obtain the usual, \textit{non-twisted} definitions of the Selberg and Ruelle zeta functions (see \cite{Rue76,sel}).
	
	We summarize here the main results of the paper:
	
	\begin{thmalph}[see Theorems~\ref{selbergmero} and \ref{selbergfe}]\label{thm:IntroThm}
		Let $X=\Gamma\backslash \H^2$ be a compact hyperbolic surface and let $\chi\colon \Gamma\rightarrow \GL(V_{\chi})$ be a finite-dimensional, complex representation of $\Gamma$. Then, the twisted Selberg zeta function $Z(s;\chi)$ admits a meromorphic continuation to $\C$ which satisfies the functional equation \eqref{fefinal}.
	\end{thmalph}
	
	Using the fundamental relation $R(s;\chi)=Z(s;\chi)/Z(s+1;\chi)$ between the Selberg and the Ruelle zeta function, we immediately obtain:
	
	\begin{coralph}[see Corollary~\ref{cor:RuelleMeromorphic}]\label{cor:IntroRuelleMero}
		The twisted Ruelle zeta function $R(s;\chi)$ admits a meromorphic continuation to $\C$ which satisfies the functional equation
		$$ R(s;\chi)R(-s;\chi)=(2\sin\pi s)^{2(2g-2)\dim V_\chi}. $$
	\end{coralph}
	
	The functional equation for the Ruelle zeta function immediately implies:
	
	\begin{coralph}[see Corollary~\ref{cor:RuelleAtZero}]\label{cor:IntroRuelleZero}
		The behavior of the twisted Ruelle zeta function $R(s;\chi)$ near $s=0$	is given by
		\begin{equation*}
			R(s;\chi)=\pm(2\pi s)^{\dim(V_{\chi})(2g-2)}+\text{higher order terms}.
		\end{equation*}
	\end{coralph}

	The sign in Corollary~\ref{cor:IntroRuelleZero} is related to the multiplicity of the zero eigenvalue of a twisted Laplacian on $X$, see Remark~\ref{rem:SignInRuelleAtZero}.

	\paragraph{Relation to previous results.}
	
	In \cite[Theorem 1 and Corollary 2]{fried1986fuchsian}, Fried showed Corollary~\ref{cor:IntroRuelleMero} and \ref{cor:IntroRuelleZero} in the special case where the representation $\chi$ of $\pi_{1}(X)$ is \emph{unitary}, building upon the Selberg trace formula developed by Hejhal~\cite{Hej76}. The purpose of this paper is to obtain an explicit resolvent trace formula and use it to generalize Fried's results to the case of \emph{non-unitary} representations of the fundamental group. Let us remark that Dyatlov and Zworski in \cite[Theorem p. 212]{dyatlov2017ruelle} generalized Fried's results to the case of (non-twisted) Ruelle zeta functions for oriented surfaces with (not necessarily constant) negative curvature, using techniques from semiclassical analysis. Concerning $2$-dimensional manifolds, we mention here also the recent paper of Riviere and Dang~\cite{dang2020poincar}, where an object closely related to the dynamical zeta functions, the generalized Poincar\'{e} series is defined and its value at $s=0$ studied.

	For higher-dimensional compact hyperbolic manifolds, Fried computed in \cite[Theorem 3]{Fried} the order of vanishing of the Ruelle zeta function twisted by a unitary representation in terms of Betti numbers (see also the recent work \cite{CDKP20} for perturbations of the hyperbolic metric on $3$-manifolds). In the odd-dimensional case it turns out that the leading coefficient is in fact related to a topological invariant, the Reidemeister torsion (\cite{DR,franz1935torsion,reidemeister1935homotopieringe}). This led Fried to conjecture a precise relation between the value $R(0;\chi)$ and the Reidemeister torsion for all compact locally symmetric spaces and acyclic representations $\chi$. For unitary representations $\chi$, Fried's conjecture was recently proven by Shen~\cite{Shen2018} (see also \cite{BO,BO99,CD19,dang2019fried,MS91,MSF,mueller2020fried,shen2020complex,shen2021analytic,Spilioti2018,spilioti2020functional,spilioti2020twisted,Wo} for other contributions to Fried's conjecture).

	In fact, also for quotients $X=\Gamma\backslash\H^2$ of the upper half plane, Fried obtained in \cite{fried1986fuchsian} a relation between the value of a Ruelle zeta function at $s=0$ and a topological invariant. He considered a general Fuchsian group $\Gamma$, allowing also elliptic elements. In such a case, $X$ is a compact orbisurface and the unit sphere bundle $S(X)=\Gamma\backslash\PSL_{2}(\R)$ is a Seifert fiber space over $X$. Fried considered the twisted Ruelle zeta function $R(s;\rho)$, associated with a unitary representation $\rho$ of $\pi_{1}(S(X))$. Note that $\pi_{1}(X)\simeq \pi_{1}(S(X))/ \Z$. By \cite[Theorem 3]{fried1986fuchsian}, for an acyclic, unitary representation $\rho$ of $\pi_{1}(S(X))$, the value $\vert R(0;\rho)\vert^{-1}$ equals the Reidemeister torsion of $S(X)$ in the representation $\rho$. It is natural to ask whether one can extend these results to the case of a non-unitary representation of $\pi_{1}(S(X))$. We deal with this problem the subsequent work~\cite{BFS21}.
	
	\paragraph{Outline of the proof.}

	The main tool to prove the above results is a twisted Selberg trace formula for non-unitary representations of $\Gamma$ (see Theorem~\ref{thm:TraceFormula}). This trace formula holds for the heat operators of a twisted Laplace operator $\Delta_{\chi}^{\sharp}$ on the flat vector bundle $E_\chi$ over $X$, associated with $\chi$. This twisted Laplacian is not self-adjoint in general, but it is elliptic with self-adjoint principle symbol and hence it has similar spectral properties (for further details see Subsections~\ref{sec:TwistedBLOperator} and \ref{sec:TraceFormula}). A slightly weaker version of this trace formula, namely for Paley--Wiener functions applied to $(\Delta_{\chi}^{\sharp})^{1/2}$ instead of heat operators, was previously obtained by M\"{u}ller~\cite[Theorem 1.1]{M1} in the more general context of rank one locally symmetric spaces. Our trace formula can, after some additional effort estimating the heat kernel, be used to obtain a resolvent trace formula (see Proposition~\ref{resolventtrace}) which is the key to prove Theorem~\ref{thm:IntroThm}. We note that the trace formula and the functional equation for the twisted Selberg zeta function that we obtain generalize the corresponding results in Hejhal~\cite[Chapter III, Theorem 4.10]{Hej76}.
	
	Finally, Corollary~\ref{cor:IntroRuelleZero} on the behavior of the twisted Ruelle zeta function follows from a detailed analysis of the functional equation for the twisted Selberg zeta function. Since the functional equation is the same as in the case of a unitary representation of $\Gamma$, the same computations as in \cite{fried1986fuchsian} apply.	
	
	\paragraph{Organization of the paper.}
	
	In Section~\ref{sec:Preliminaries}, we review well-known theory
	of the geometry of hyperbolic surfaces, the Laplace--Beltrami operator, and
	the principal series representation of $\PSL_{2}(\R)$.
	In Section~\ref{sec:HarmonicAnalysis}, we introduce the twisted Bochner--Laplace operator acting on the space of sections of a vector bundle over $X$ and recall its spectral properties.
	In the same section, we derive the resolvent trace formula for the heat operator induced by the twisted Laplacian.
	Section~\ref{sec:ZetaFunctions} is the core of this article.
	In this section, we define the Selberg and Ruelle zeta function associated 
	with a non-unitary representation of the fundamental group of $X$ and prove that
	they admit a meromorphic continuation to $\C$.  Moreover, we prove functional
	equations for the twisted Selberg function, which further  
	lead to the result for the behavior of the twisted Ruelle zeta function at the origin.
	
	\paragraph{Acknowledgements.}
	
	The authors would like to thank L\'{e}o B\'{e}nard and Werner M\"{u}ller
	for helpful discussions and comments about the results of Fried in \cite{fried1986fuchsian}. We are also grateful to the referee for various helpful comments that helped to improve the paper, one of them making the functional equation of the twisted Ruelle zeta function more explicit.
	Both authors were supported by a research grant from the Villum Foundation (Grant No. 00025373).

	\section{Preliminaries}\label{sec:Preliminaries}
	
	In this section we set up some notation around the geometry and analysis on hyperbolic surfaces and representation theory of $\PSL(2,\R)$.
	
	\subsection{Hyperbolic surfaces}\label{sec:HyperbolicSurfaces}
	We consider the upper half plane 
	\[
	\H^2:=\{z=x+iy:y>0\}.
	\]
	The group $G=\PSL(2,\R)$ acts on $\H^2$ by fractional linear transformations. This action is transitive and the maximal compact subgroup $K=\PSO(2)$ is the stabilizer of $i\in\H^2$, hence $\H^2\cong G/K=\PSL(2,\R)/\PSO(2)$.
	Let  $\mathfrak{g}$ and $\mathfrak{k}$ denote the Lie algebras of $G$ and $K$. The $G$-invariant metric, induced by the restriction of the Killing form on $\mathfrak{g}$ to the Cartan complement $\mathfrak{p}$ of $\mathfrak{k}$, is the Poincar\'{e} metric
	\begin{equation}\label{poinc}
		ds^2=\frac{dx^2+dy^2}{y^2}.
	\end{equation}
	Let $G=NAK$ be the Iwasawa decomposition of $G$ with
	\begin{equation*}
		N=\bigg\{
		\begin{pmatrix}
			1 & x\\
			0 & 1
		\end{pmatrix}: x\in \R \bigg\} \qquad \mbox{and} \qquad
		A=\bigg\{
		\begin{pmatrix}
			y^{1/2} & 0\\
			0 & y^{-1/2}
		\end{pmatrix}: y>0 \bigg\},
	\end{equation*}
	where matrices are identified with their images in the quotient group $\PSL(2,\R)=\SL(2,\R)/\{\pm I\}$. Let $A^+\subseteq A$ be the subset of all elements of the above form with $y>1$. We write $\mathfrak{n}$ and $\mathfrak{a}$ for the Lie algebras of $N$ and $A$.
	
	Now, let $\Gamma$ be a Fuchsian group, that is a discrete subgroup of $G$. We assume that $\Gamma$ is cocompact, i.e., $\Gamma\backslash G$ is compact, and that $\Gamma$ is torsion-free, i.e., there are no non-trivial elements of finite order. Then, $X=\Gamma\backslash \H^2=\Gamma\backslash G/K$	is an orientable compact connected hyperbolic surface, and conversely every orientable compact connected hyperbolic surface arises in this way.
	
	In this setting, every $\gamma\in\Gamma$ with $\gamma\neq e$, where $e$	is the identity element, is hyperbolic and hence conjugate to a unique element $a_\gamma\in A^+$. It follows that the non-trivial conjugacy classes $[\gamma]$ in $\Gamma$ parameterize the set of closed geodesics on $X$. Here a geodesic is called \emph{closed} if it returns to its starting point with the same direction. The length $l(\gamma)$ of the closed geodesic corresponding to a non-trivial conjugacy class $[\gamma]$ is related to the element $a_\gamma$ by $\Ad(a_\gamma)|_{\mathfrak{n}}=e^{l(\gamma)}\Id_{\mathfrak{n}}$. A non-trivial conjugacy class $[\gamma]$ in $\Gamma$ is called \emph{primitive} if $\gamma$ cannot be written as $\gamma=\gamma_0^n$, with $\gamma_0\in\Gamma$ and $n\geq2$. Then $[\gamma]$ is primitive if and only if the corresponding geodesic is \emph{prime}, i.e., it cannot be written as the multiple of a shorter closed geodesic. For every non-trivial conjugacy class $[\gamma]$ we can write $\gamma=\gamma_0^n$ with $[\gamma_0]$ primitive. The number $n=n_\Gamma(\gamma)$ is unique, independent of the representative $\gamma$ and satisfies $l(\gamma)=n\cdot l(\gamma_0)$.

	\subsection{The Laplace--Beltrami operator}
	
	Let
	\begin{equation}\label{Lb2}
		\widetilde{\Delta}=-y^{2}\bigg(\frac{\partial^{2}}{\partial x}
		+ \frac{\partial^{2}}{\partial y}\bigg),
	\end{equation}
	denote the Laplace--Beltrami operator on $\H^2$, where we write $z=x+iy\in \H^{2}$. If we define the left translation of a smooth function $f$ on $\H^{2}$ by $g\in G$ as
	\[
	L_{g}f(z):=f(g^{-1}z),
	\]
	then, for every $g\in G$
	\[
	L_{g}\circ\widetilde{\Delta}=\widetilde{\Delta}\circ L_{g}.
	\]
	Hence, $\widetilde{\Delta}$ descends to a differential operator $\Delta$ on $\Gamma\backslash\H^{2}$. 

	\subsection{Unitary principal series representations}

	Let $P=AN$ denote the standard parabolic subgroup of $G$. We identify $\mathfrak{a}_\C^*\simeq\C$ by $\lambda\mapsto\lambda(\operatorname{diag}(\frac{1}{2},-\frac{1}{2}))$. Then $\rho=\frac{1}{2}\tr\ad|_{\mathfrak{n}}$ corresponds to $\frac{1}{2}$. For $\lambda\in\R$ we form the character $e^{i\lambda}$ of $A$ given by $e^{i\lambda}(e^H)=e^{i\lambda(H)}$ ($H\in\mathfrak{a}$), and extend it to a character $e^{i\lambda}\otimes1$ of $P=AN$ by letting $N$ act trivially. We form the induced representation
	$$ \pi_\lambda = \operatorname{Ind}_P^G(e^{i\lambda}\otimes 1) $$
	of $G$ on the Hilbert space $\mathcal{H}_\lambda$ of $L^2$-sections of the homogeneous vector bundle $G\times_P(e^{i\lambda+\rho}\otimes1)\to G/P$. The representations $\pi_\lambda$ are unitary and irreducible for all $\lambda\in\R$ and are called the \emph{unitary principal series}. We write $\Theta_\lambda$ for the \textit{distribution character} of $\pi_\lambda$ defined by
	$$ \Theta_\lambda(\varphi) = \tr\int_G\varphi(g)\pi_\lambda(g)\,dg \qquad (\varphi\in C_c^\infty(G)), $$
	where $dg$ denotes a suitably normalized Haar measure on $G$.
	
We denote by $\mathcal{H}_\lambda^{K}$ the one-dimensional subspace of $\mathcal{H}_\lambda$ consisting of $K$-fixed vectors.

	\section{Harmonic analysis on hyperbolic surfaces}\label{sec:HarmonicAnalysis}
	
	Following M\"{u}ller~\cite{M1}, we describe the construction of the twisted Bochner--Laplace operator on vector bundles over $X=\Gamma\backslash\H^2$ induced from non-unitary representations $\chi$ of $\Gamma$ and use it to obtain a trace formula for $X$.
	
	\subsection{The twisted Bochner--Laplace operator}\label{sec:TwistedBLOperator}
	
	In this section, we define the twisted Bochner--Laplace operator as introduced  in \cite[Section 4]{M1}. This operator acts on sections of a vector bundle over $X$. It is an elliptic operator which is in general not self-adjoint. Nevertheless, it has a self-adjoint principle symbol (with respect to the choice of a metric) and hence it has qualitatively similar spectral properties.
	
Let $\chi$ be a finite-dimensional, complex representation
	\begin{equation*}
		\chi\colon\Gamma\rightarrow \GL(V_{\chi})
	\end{equation*}
	of $\Gamma$.
 Denote by $E_{\chi}=V_\chi\times_\Gamma\H^2\rightarrow X$ the associated flat vector bundle 
 over $X$, equipped with a flat connection $\nabla^{E_{\chi}}$.
We recall the construction of the twisted Bochner--Laplace operator $\Delta_\chi^\sharp$, acting on smooth sections of $E_{\chi}$.

The second covariant derivative $(\nabla^{E_\chi})^2$ is defined by
	$$
	(\nabla^{E_\chi})^2_{V,W}:=\nabla_{V}^{E_\chi}\nabla_{W}^{E_\chi}-\nabla^{E_\chi}_{\nabla^{LC}_{V}W}, 
$$
where $V,W\in C^{\infty}(X,TX)$, $TX$ is the tangent bundle of $X$,
and $\nabla^{LC}$ denotes the Levi--Civita connection on $TX$. 
The twisted Bochner--Laplace operator $\Delta_\chi^\sharp$ is defined to be the corresponding connection Laplacian on $E_\chi$, i.e. the negative of the trace of the second covariant derivative:
	\begin{equation}\label{eq:DefDeltaSharp}
	\Delta_\chi^\sharp := -\tr\big((\nabla^{E_\chi})^2\big).
	\end{equation}

\begin{rmrk}
The definition of the twisted Bochner--Laplace operator in the present paper is the same as in 
\cite[Section 4]{M1} for $\tau$ being the trivial representation.
\end{rmrk}

Locally, this operator is described as follows.
We consider an open subset $U$ of $X$ such that $E_{\chi}\lvert_{U}$ is trivial, 
i.e.,  $E_\chi\lvert_U\cong U\times\C^m$, where $m=\rank(E_\chi)=\dim V_{\chi}$.
Let $e_1,\ldots,e_m$ be any basis of flat sections of $E_{\chi}\lvert_{U}$. Then, each 
$\phi \in C^{\infty}(U,E_{\chi}\lvert_{U})$ can be written as
	$
	\phi=\sum_{i=1}^{m}\phi_{i} \otimes e_{i},
	$
	where $\phi_{i}\in C^{\infty}(U)$, $i=1,\ldots,m$. Then, 
	\begin{equation}
		\nabla_{Y}^{E_{\chi}}\phi=\sum_{i=1}^{m}\nabla_{Y}\phi_{i}\otimes e_{i}, \qquad (Y\in C^{\infty}(X,TX)).\label{eq:ConnectionLocally}
	\end{equation}
	The local expression above is independent of the choice of the basis of flat sections of $E_{\chi}\lvert_{U}$, since the transition maps comparing flat sections are constant. By {\eqref{eq:DefDeltaSharp}} and \eqref{eq:ConnectionLocally}, the twisted Bochner--Laplace operator acting on $C^{\infty}(U,E_{\chi}\lvert_{U})$ is given by
	\begin{equation}\label{sharploc}
		\Delta_\chi^\sharp\phi=\sum_{i=1}^{m}(\Delta\phi_{i})\otimes e_{i},
	\end{equation}
	where $\Delta$ denotes the Laplace--Beltrami operator on $X$.

	Let now $\widetilde{E}_{\chi}$ be the pullback to $\widetilde{X}=\H^2$ of $E_{\chi}$.
	Then,
	\begin{equation*}
		\widetilde{E}_{\chi}\cong \widetilde{X}\times V_{\chi},
	\end{equation*}
	and
	\begin{equation}\label{iso}
		C^{\infty}(\widetilde{X},\widetilde{E}_{\chi})\cong  C^{\infty}(\widetilde{X})\otimes V_{\chi}.
	\end{equation}
	With respect to the isomorphism {\eqref{iso}}, it follows from {\eqref{sharploc}} that the lift $\widetilde{\Delta}^{\sharp}_\chi$ of $\Delta_\chi^{\sharp}$ to $\widetilde{E}_{\chi}$ takes the form
	\begin{equation}\label{eq:DeltaSharpTilde}
		\widetilde{\Delta}^{\sharp}_\chi = \widetilde{\Delta}\otimes\Id_{V_{\chi}},
	\end{equation}
	where $\widetilde{\Delta}$ is the Laplace--Beltrami operator on $\H^2$.
	
	If we choose a Hermitian metric on $E_{\chi}$, then $\Delta_\chi^{\sharp}$ acts in $L^{2}(X,E_{\chi})$ with domain $C^{\infty}(X,E_\chi)$. However, it is not a formally self-adjoint operator in general. By {\eqref{sharploc}}, $\Delta_\chi^{\sharp}$ has principal symbol 
	\begin{equation*}
		\sigma_{\Delta_\chi^\sharp}(x,\xi)=\lVert \xi \rVert^ {2}_{x} \Id_{({E_{\chi})_{x}}}
		\qquad (x\in X, \xi\in T_{x}^{*}X).
	\end{equation*}
	Hence, $\Delta_\chi^\sharp$ is an elliptic, second order differential 
	operator with the following spectral properties:
	its spectrum is discrete and contained in a translate of a positive cone $C\subset \C$ such that $\R^{+}\subset C$.
	This fact follows from classical spectral theory of elliptic 
	operators, under the assumption of the compactness of the manifold.
	We refer the reader to \cite{Sh}, and also \cite[Lemma 2.1]{M1}. Moreover, the direct sum of all generalized eigenspaces is dense in $L^2(X,E_\chi)$.

	\subsection{The trace formula}\label{sec:TraceFormula}
	
	For the derivation of the trace formula, we follow the classical work of Wallach~\cite{Wab}, and its extension to non-unitary representations of $\Gamma$ by M\"{u}ller~\cite{M1}. This trace formula will be the basic tool to prove the meromorphic continuation of the twisted dynamical zeta functions in the next section.
	
	We denote by $\spec(\Delta^{\sharp}_{\chi})\subseteq\C$ the (discrete) spectrum of $\Delta^{\sharp}_{\chi}$. For $\mu\in    
     \spec(\Delta^{\sharp}_{\chi})$, we write $L^2(X,E_\chi)_\mu$ for the corresponding generalized eigenspace.
	We define the algebraic multiplicity $m(\mu)$ of $\mu$ as $m(\mu):=\dim L^2(X,E_\chi)_\mu$.
	
	We want to utilize the heat operator $e^{-t{\Delta}^{\sharp}_{\chi}}$, induced by $\Delta^{\sharp}_{\chi}$,
	as an integral, trace class operator and derive a corresponding trace formula. In \cite{M1}, a Selberg trace formula
	for non-unitary twists is derived for particular test functions $\phi$, the Paley--Wiener functions on $\C$ (see \cite[p. 2079]{M1}). Since the exponential function $\phi(\lambda)=e^{-t\lambda^2}$, $\lambda\in \C$ does not belong to this space, we use the extended results obtained in \cite[p. 171--173]{Spilioti2018}. With this, we conclude
	that the heat operator $e^{-t{\Delta}^{\sharp}_{\chi}}$
	is an integral operator with smooth kernel,
	i.e., there exists a smooth section $H^{\chi}_{t}$ of $\End(E_\chi)$ on
	on $X\times X$ such that for $f\in L^{2}(X,E_{\chi})$, we have
	\begin{equation*}
		e^{-t{\Delta}^{\sharp}_{\chi}}f(x)=\int_{X}H^{\chi}_{t}(x,y)f(y)dy.
	\end{equation*}
	By \cite[Proposition 2.5]{M1}, this operator is of trace class.
	By Lidskii's theorem~\cite[Theorem 3.7]{SB}, we have 
	\begin{equation}
		\tr(e^{-t{\Delta}^{\sharp}_{\chi}})=\sum_{\mu\in
			\spec(\Delta^{\sharp}_{\chi})}m(\mu)e^{-t\mu},
	\end{equation}
	which is the spectral side of our (pre)-trace formula.
	Let $H_{t}$ be the kernel of the heat operator
	$e^{-t\widetilde{\Delta}}$, that is the operator induced by the self-adjoint Laplacian $\widetilde{\Delta}$ acting in $L^2(\widetilde{X})$. By \cite[Lemma 2.3 and Proposition 2.4]{BM}, $H_t$ is contained in the Harish-Chandra $L^q$-Schwartz space $\mathcal{C}^q(G)$ for any $q>0$ (see e.g. \cite[p. 161--162]{BM} for the definition of the Schwartz space). Moreover, as in \cite[p. 174]{Spilioti2018}, it follows from \eqref{eq:DeltaSharpTilde} that
	\begin{equation}
		H_{t}^\chi(x,x')=\sum_{\gamma \in \Gamma}H_{t}(g^{-1}\gamma g')\chi(\gamma),\label{eq:HeatKernelAsSumOverGamma}
	\end{equation}
	where $x=\Gamma g K,x'=\Gamma g' K\in X$ with $g,g'\in G$. By \cite[Proposition 5.3]{Spilioti2018}, we have the following result.
	
	\begin{prop}
		Let $E_{\chi}$ be a flat vector bundle over $X=\Gamma\backslash \widetilde{X}$ associated with a finite-dimensional, complex
		representation $\chi\colon\Gamma\rightarrow \GL(V_{\chi})$ of $\Gamma$. Let $\Delta_{\chi}^{\sharp}$ be the twisted Bochner--Laplace operator acting in $L^2(X,E_{\chi})$. Then, 
		\begin{equation}
			\sum_{\mu\in\spec(\Delta_{\chi}^\sharp)}m(\mu)e^{-t\mu} = \tr(e^{-t\Delta_{\chi}^{\sharp}}) =\int_{\Gamma\backslash G}\left(\sum_{\gamma \in \Gamma}\tr \chi(\gamma)\cdot H_{t}(g^{-1}\gamma g)\right)\,d\dot{g}.
		\end{equation}
	\end{prop}
	
	As in \cite[Proposition 5.1, Proposition 6.1]{M1} and \cite[p. 172--173, 177--178]{Wa}, we group the summation into conjugacy classes and use the Fourier inversion formula to obtain
	\begin{multline}\label{tracewallach}
		\tr(e^{-t\Delta_{\chi}^{\sharp}})=\dim(V_{\chi})\Vol(X)H_{t}(e)\\
		+\sum_{[\gamma]\neq e} \tr\chi(\gamma)\frac{l(\gamma)}{n_{\Gamma}(\gamma)D(\gamma)}\frac{1}{2\pi}
		\int_{\R}\Theta_\lambda(H_{t})e^{-il(\gamma)\lambda}d\lambda,
	\end{multline}
	where
	$$ D(\gamma) = e^{-\frac{l(\gamma)}{2}} |\det(\Ad(a_{\gamma})|_{\mathfrak{n}}-\Id)|. $$
	Note that every $\gamma\in\Gamma$, $\gamma\neq e$, is hyperbolic, so the summation is over hyperbolic conjugacy classes of $\Gamma$. We therefore refer to the first term in the $\RHS$ of \eqref{tracewallach}
	as the \emph{identity contribution} and to the second term as the \emph{hyperbolic contribution}.
	
	\begin{rmrk}
		The trace formula in \cite{Wa} is obtained for kernel functions that are compactly supported which is not true for the heat kernel. However, the heat kernel is an admissible function in the sense of Gangolli~\cite[p. 407]{gangolli1977}. This is because the heat kernel belongs to the Harish-Chandra $L^{1}$-Schwartz space $\mathcal{C}^{1}(G)$ (see \cite[Proposition 3.1 and p. 411]{gangolli1977}).
	\end{rmrk}
	
	
	Let us first simplify the hyperbolic contribution. The character $\Theta_\lambda$ of $\pi_\lambda$ can be evaluated on the $K$-biinvariant function $H_t$ in terms of the spherical Fourier transform $\widetilde{H_{t}}(\lambda)$ of $H_{t}$:
	\begin{equation*}
		\Theta_{\lambda}(H_{t}) = \tr\pi_{\lambda}(H_{t}) = \int_{G} \langle\pi_{\lambda}(g)v,v\rangle H_{t}(g)dg = \int_{G} \phi_\lambda(g)H_{t}(g)dg = \widetilde{H_{t}}(\lambda),
	\end{equation*}
	where $v\in \mathcal{H}_\lambda^K$, with $\lVert v \rVert=1$ and $\phi_\lambda(g)=\langle\pi_{\lambda}(g)v,v\rangle$ denotes the associated spherical function (see e.g. \cite[Chapter IV]{Hel} for details).
	
	\begin{lem}
		\begin{equation}\label{character}
			\Theta_{\lambda}(H_{t})=e^{-t(\lambda^2+\frac{1}{4})}.
		\end{equation}
	\end{lem}
	
	\begin{proof}
		By \cite[eq. (2.11)]{BM}, 
		\begin{equation*}
			\widetilde{H_t}(\lambda)=e^{t\pi_\lambda(\Omega)},
		\end{equation*}
		where $\Omega$ is the Casimir element in the center of the universal enveloping algebra of $\mathfrak{sl}(2,\C)$.
		By \cite[Proposition 8.22]{Knapp}, 
		\begin{equation*}
			\pi(\Omega)=-\lambda^2-\rho^2.
		\end{equation*}
		Recall now from Subsection 2.3 that $\rho$ corresponds to $\frac{1}{2}$.
	\end{proof}
	
	We now turn to the identity contribution in \eqref{tracewallach}. The function $H_{t}$ on $G$ belongs to the Harish-Chandra $L^{q}$-Schwartz space. Hence, the Fourier inversion formula \cite[Theorem 3]{HC} can be applied to $H_{t}$
	(see also Theorem 1 in  \cite{ANKER1991331} for $K$-biinvariant test functions in the Harish-Chandra $L^{q}$-Schwartz space for $0<q\leq2$).
	By \cite[Theorem~7.5~(i)]{Hel} (see also \cite[eq. (28) and (29), p. 42]{Hel}), we have
	\begin{equation}\label{planch2a}
		H_{t}(e)=\frac{1}{4\pi^2}\int_{\R}\Theta_{\lambda}(H_{t})\rvert c(\lambda)\lvert^{-2} d\lambda
	\end{equation}
	with the $c$-function $c(\lambda)$ given by (note the change of variables $\lambda$ vs. $\frac{\lambda}{2}$ compared to \cite{Hel})
	\begin{equation}\label{planch2b}
		\lvert c(\lambda)\lvert^{-2}=\lambda\pi\tanh\lambda\pi,\quad \lambda\in\R.
	\end{equation}
	
	\begin{thm}[Trace formula]\label{thm:TraceFormula}
		Let $\chi\colon\Gamma\rightarrow\GL(V_{\chi})$ be a finite-dimensional representation of $\Gamma$. 
		Then, the following Selberg trace formula for the operator $e^{-t\Delta_{\chi}^{\sharp}}$ holds:
		\begin{multline}\label{selberg2}
			\tr(e^{-t\Delta_{\chi}^{\sharp}})=\frac{1}{4\pi^2}\dim(V_{\chi})\Vol(X)\int_{\R}e^{-t(\lambda^2+\frac{1}{4})}\lambda\pi\tanh\lambda\pi d\lambda\\
			+\sum_{[\gamma]\neq e} \tr\chi(\gamma)\frac{l(\gamma)}{n_{\Gamma}(\gamma)D(\gamma)}\frac{1}{2\pi}
			\int_{\R}e^{-t(\lambda^2+\frac{1}{4})}e^{-il(\gamma)\lambda}d\lambda.
		\end{multline}
	\end{thm}
	
	\begin{proof}
		The trace formula \eqref{selberg2} follows from \eqref{tracewallach}, \eqref{character}, \eqref{planch2a} and the explicit formula for the $c$-function \eqref{planch2b}.
	\end{proof}

	\section{Twisted dynamical zeta functions on compact hyperbolic surfaces}\label{sec:ZetaFunctions}
	
	In this section we prove meromorphic continuation of the twisted Selberg and the Ruelle zeta function, obtain their functional equations and study the behaviour of the twisted Ruelle zeta function at $s=0$.
	
	\subsection{Definition and convergence}
	
	Let $\chi\colon \Gamma\rightarrow \GL(V_{\chi})$ be a finite-dimensional, complex representation of $\Gamma$. For $s\in\C$,  we define the twisted Selberg zeta function, associated with $\chi$, by the infinite product
	\begin{equation}\label{selbergzeta}
		Z(s;\chi):=\prod_{\substack{[\gamma]\neq{e}\\ [\gamma]\prim}} \prod_{k=0}^{\infty}\det\big(\Id-\chi(\gamma)e^{-(s+k)l(\gamma)}\big).
	\end{equation}
	It is proven in \cite[Theorem 3.1]{fedosova2020meromorphic} that the product in \eqref{selbergzeta} converges absolutely for $s$ in some right half plane in $\C$. To keep the paper self-contained, and also because we need parts of the computation later on, we include a proof of this result:
	
	\begin{prop}\label{SelbergConvergence}
		There exists a positive constant $c_1$ such that the product \eqref{selbergzeta} defining the twisted Selberg zeta function converges absolutely and uniformly on compact subsets of the half-plane $\RE(s)>c_{1}$.
	\end{prop}
	\begin{proof}
	Let $\Vert\cdot\Vert$ be the operator norm associated to a fixed norm on $V_{\chi}$.
	By \cite[p. 10]{Wo}, there exists a $c\geq 0$ such that 
	\begin{equation}\label{chibound}
	\Vert\chi(\gamma)\Vert\leq e^{cl(\gamma)}. 
	\end{equation}
	Hence, for $\RE(s)\gg 0$, by \eqref{selbergzeta}, we get
		\begin{align}\label{logselbergzeta}
			\log Z(s;\chi)\notag=&\sum_{\substack{[\gamma]\neq{e}\\ [\gamma]\prim}} \sum_{k=0}^{\infty}\tr\log(1-\chi(\gamma)e^{-(s+k)l(\gamma)})\\\notag
			=&-\sum_{\substack{[\gamma]\neq{e}\\ [\gamma]\prim}}\sum_{k=0}^{\infty}\sum_{j=1}^{\infty}\frac{\tr((\chi(\gamma)e^{-(s+k)l(\gamma)})^{j})}{j}\\\notag
			=&-\sum_{[\gamma]\neq{e}}\sum_{k=0}^{\infty} \frac{1}{n_{\Gamma}(\gamma)}\tr(\chi(\gamma))e^{-(s+k)l(\gamma)}\\
			=&-\sum_{[\gamma]\neq{e}}\frac{1}{n_{\Gamma}(\gamma)}\tr(\chi(\gamma))\frac{e^{-sl(\gamma)}}{1-e^{-l(\gamma)}}.
		\end{align}
		By the inequality $\vert \tr(\chi(\gamma))\vert \leq \dim(V_{\chi})\Vert\chi(\gamma)\Vert$
		and \eqref{chibound}, we get 
	\begin{equation}\label{eq:EstimateTraceChi}
	\vert \tr(\chi(\gamma))\vert \leq c'e^{cl(\gamma)}
	\end{equation}
	(see also \cite[Lemma 3.3]{Spilioti2018}).
		Moreover, if we define
		\begin{equation*}
			\mathcal{N}_\Gamma^C(R):=\sharp\{[\gamma] \in C( \Gamma)\colon l(\gamma)\leq R\}\qquad (R\geq 0),
		\end{equation*}
		where $C(\Gamma)$ denotes the set of $\Gamma$-conjugacy classes, then by \cite[equation (1.31)]{BO} there exists a positive constant $C$ such that
		\begin{equation}\label{lenghtgrowthestimates}
			\mathcal{N}^{C}_{\Gamma}(R)\leq C e^{R}.
		\end{equation}
		Hence, the assertion follows from \eqref{logselbergzeta}, \eqref{eq:EstimateTraceChi} and \eqref{lenghtgrowthestimates}.
	\end{proof}
	
	\begin{lem}\label{selberglog}
		Let $L(s;\chi):=\frac{d}{ds}\log Z(s;\chi)$ be the logarithmic derivative of $Z(s;\chi)$. Then,
		\begin{equation}\label{logder}
			L(s;\chi)= \sum_{[\gamma]\neq{e}}\frac{l(\gamma)\tr(\chi(\gamma))}{2n_{\Gamma}(\gamma)\sinh(l(\gamma)/2)}e^{-(s-\frac{1}{2})l(\gamma)}.
		\end{equation}
	\end{lem}
	\begin{proof}
		\eqref{logder} follows easily by differentiating \eqref{logselbergzeta}.
	\end{proof}
	
	We define the twisted Ruelle zeta function, associated with $\chi$,  by the infinite product
	\begin{equation}\label{ruellezeta}
		R(s;\chi)=\prod_{\substack{[\gamma]\neq{e}\\ [\gamma]\prim}}\det (\Id-\chi(\gamma)e^{-sl(\gamma)}).
	\end{equation}

	\begin{prop}\label{prop:RuelleConvergence}
		There exists a positive constant $c_2$ such that the product \eqref{ruellezeta} defining the twisted Ruelle zeta 
		function converges absolutely and uniformly on compact subsets of the half-plane $\RE(s)>c_{2}$.
	\end{prop}
	\begin{proof}
		The proof is similar to the proof of Proposition~\ref{SelbergConvergence} since
		\begin{align}\label{logruellezeta}
			\log R(s;\chi)\notag=&\sum_{\substack{[\gamma]\neq{e}\\ [\gamma]\prim}} \tr\log(1-\chi(\gamma)e^{-sl(\gamma)})\\
			=&-\sum_{\substack{[\gamma]\neq{e}\\ [\gamma]\prim}}\sum_{j=1}^{\infty}\frac{\tr((\chi(\gamma)e^{-sl(\gamma)})^{j})}{j}\\=&-\sum_{[\gamma]\neq{e}}\frac{1}{n_{\Gamma}(\gamma)}\tr(\chi(\gamma))e^{-sl(\gamma)}.\qedhere
		\end{align}
	\end{proof}
	
	\begin{rmrk}
	The constants $c_{1}, c_{2}$ in Propositions \ref{selberglog} and \ref{prop:RuelleConvergence},
	respectively, can be explicitly determined. In fact, we can chose $c_{1}=c_{2}=c+1$,
	 where $c$ is as in \eqref{chibound}. 
	We remark also that one can consider the notion of the critical exponent of $
	\chi$, in order to eliminate the dependence on the norm $\Vert \cdot\Vert$ on $V_{\chi}$
	in the estimates \eqref{chibound}. The critical exponent $c_{\chi}$ is defined as
	$c_{\chi}:=\inf\{c\geq 0, \text{such that \eqref{chibound} holds}\}$.
	Note that for $\chi$ being unitary, we have that $c_{\chi}=0$.
	For more details we refer the reader to \cite[p. 10]{Wo}.
	\end{rmrk}
	
	\begin{lem}\label{ruelleselberg}
		The twisted Ruelle zeta function and the twisted Selberg zeta function are related by the following identity:
		\begin{equation}\label{ruellesel}
			R(s;\chi)=\frac{Z(s;\chi)}{Z(s+1;\chi)}.
		\end{equation}
	\end{lem}
	\begin{proof}
		\eqref{ruellesel} follows by considering \eqref{logselbergzeta} at $s+1$
		and \eqref{logruellezeta}.
	\end{proof}
	

	\subsection{Meromorphic continuation}
	
	The trace formula \eqref{selberg2} motivates to consider a shift of the operator $\Delta_{\chi}^{\sharp}$ by $\frac{1}{4}$:
	\begin{equation*}
		A^{\sharp}_{\chi}:=\Delta_{\chi}^{\sharp}-\frac{1}{4}.
	\end{equation*}
	We now explicitly calculate the hyperbolic contribution on the $\RHS$ of \eqref{selberg2}.
	
	\begin{coro}
		Let $X=\Gamma\backslash\H^2$ be a compact hyperbolic surface and $\chi$ be a finite-dimensional, complex representation of $\Gamma$. Then, we have
		\begin{multline}\label{selberg3}
			\tr(e^{-tA^{\sharp}_{\chi}})=\frac{1}{4\pi^2}\dim(V_{\chi})\Vol(X)\int_{\R}e^{-t\lambda^2}\lambda\pi \tanh\lambda\pi d\lambda\\
			+\frac{1}{2\sqrt{4\pi t}}\sum_{[\gamma]\neq e} \frac{l(\gamma)\tr(\chi(\gamma))}{n_{\Gamma}(\gamma)\sinh(l(\gamma)/2)}
			e^{-\frac{l(\gamma)^{2}}{4t}}.
		\end{multline}
	\end{coro}
	
	\begin{proof}
		Since $\Ad(a_\gamma)|_{\mathfrak{n}}=e^{l(\gamma)}\Id_{\mathfrak{n}}$, we have
		\begin{multline}\label{di}
			D(\gamma)=e^{-\frac{l(\gamma)}{2}} |\det(\Ad(a_{\gamma})|_{\mathfrak{n}}-\Id)|=e^{-l(\gamma)/2}|e^{l(\gamma)}-1|\\
			=e^{l(\gamma)/2}-e^{-l(\gamma)/2}=2\sinh(l(\gamma)/2). 
		\end{multline}
		Moreover, the integral in the hyperbolic contribution in the $\RHS$
		of \eqref{selberg2} is just the Fourier transform of the function 
		$\lambda\mapsto e^{-t\lambda^2}$. Hence, 
		substituting \eqref{di} in \eqref{selberg2}, we get
		\eqref{selberg3}.
	\end{proof}
	
	Now let $s_{1},s_{2}\in\C$ such that $(s_{1}-1/2)^2,(s_{2}-1/2)^2\in \C  \setminus\spec(-A^{\sharp}_{\chi})$.
	We consider the product of the two resolvents $(A_\chi^\sharp+(s_j-\frac{1}{2})^2)^{-1}$, for $j=1,2$.
	By \cite[Lemma 2.2]{M1}, the operator $A_\chi^\sharp$ satisfies Weyl's Law, which implies $\sum_j|\nu_j|^{-2}<\infty$, where $(\nu_j)_{j\in\Z_+}$ denote the eigenvalues of $A_\chi^\sharp$ with multiplicities. This implies that
	\begin{equation*}
		\big(A^{\sharp}_{\chi}+(s_{1}-1/2)^2\big)^{-1}\big(A^{\sharp}_{\chi}+(s_{2}-1/2)^2\big)^{-1}
	\end{equation*}
	is of trace class, and so is
	\begin{multline}\label{resid}
		\big(A^{\sharp}_{\chi}+(s_{1}-1/2)^2\big)^{-1}-(A^{\sharp}_{\chi}+(s_{2}-1/2)^2\big)^{-1}\notag\\
		=\big((s_{2}-1/2)^2-(s_{1}-1/2)^2\big)\big(A^{\sharp}_{\chi}+(s_{1}-1/2)^2\big)^{-1}\big(A^{\sharp}_{\chi}+(s_{2}-1/2)^2\big)^{-1}.
	\end{multline}
	We observe that for $\RE(s_j)\gg0$
	\begin{equation*}
		(A^{\sharp}_{\chi}+(s_j-1/2)^2)^{-1}
		=\int_{0}^{\infty}e^{-t(s_j-1/2)^2}e^{-tA^{\sharp}_{\chi}}\,dt.
	\end{equation*}
	Hence,
	\begin{multline*}
		\big(A^{\sharp}_{\chi}+(s_{1}-1/2)^2\big)^{-1}-\big(A^{\sharp}_{\chi}+(s_{2}-1/2)^2\big)^{-1}\\
		= \int_{0}^{\infty}
		\big(e^{-t(s_{1}-1/2)^2}-e^{-t(s_{2}-1/2)^2}\big)e^{-tA^{\sharp}_{\chi}}dt.
	\end{multline*}
	
	\begin{lem}[{see \cite[Chapter 2]{BGV04}}]\label{trace}
		There exist coefficients $c_{j}$ such that 
		$\tr e^{-tA^{\sharp}_{\chi}}$ has an asymptotic expansion
		\begin{equation}\label{traceeq}
			\tr e^{-tA^{\sharp}_{\chi}} \sim \sum_{j=0}^{\infty}c_{j}t^{\frac{j-2}{2}} \qquad \mbox{as }t\to0^+.
		\end{equation}
	\end{lem}

	\begin{lem}
		Let $s_{1},s_{2}\in\C$ with $\RE((s_{1}-1/2)^2),\RE((s_{1}-1/2)^2)\gg0$. Then,
		\begin{multline}\label{resolvent}
			\tr\bigg(\big(A^{\sharp}_{\chi}+(s_{1}-1/2)^2\big)^{-1}-\big(A^{\sharp}_{\chi}+(s_{2}-1/2)^2\big)^{-1}\bigg)\\
			= \int_{0}^{\infty}(e^{-t(s_{1}-1/2)^2}-e^{-t(s_{2}-1/2)^2}) \tr(e^{-tA^{\sharp}_{\chi}})dt.
		\end{multline}
	\end{lem}
	\begin{proof}
		As $t\rightarrow\infty$, the integrand in the $\RHS$ of \eqref{resolvent},
		decays exponentially. As $t\rightarrow 0$, we use Lemma \ref{trace}
		and the Taylor series expansion of $e^{-t(s_{1}-1/2)^2}$ and $e^{-t(s_{2}-1/2)^2}$ to conclude that the integral converges absolutely. Hence, we can interchange summation and integration and \eqref{resolvent} follows.
	\end{proof}

	\begin{prop}[Resolvent trace formula]\label{resolventtrace}
		Let $s_{1},s_{2}\in\C$ with $\RE(s_1),\RE(s_2)\gg0$ and $\RE((s_{1}-1/2)^2),\RE((s_2-1/2)^2)\gg0$.
		Then, 
		\begin{multline}\label{finalresolvent}
			\tr\bigg(\big(A^{\sharp}_{\chi}+(s_{1}-1/2)^2\big)^{-1}-(A^{\sharp}_{\chi}+(s_{2}-1/2)^2\big)^{-1}\bigg)\\
			=\bigg(\frac{1}{4\pi^2}\dim(V_{\chi})\Vol(X)\bigg)\bigg(\int_{\R}\frac{\lambda\pi\tanh \lambda\pi}{(s_{1}-\frac{1}{2})^{2}+\lambda^{2}}
			-\frac{\lambda\pi\tanh\lambda\pi }{(s_{2}-\frac{1}{2})^{2}+\lambda^{2}}
			d\lambda\bigg)\\
			+\frac{1}{2(s_{1}-\frac{1}{2})}L(s_{1};\chi)
			-\frac{1}{2(s_{2}-\frac{1}{2})}L(s_{2};\chi).
		\end{multline}
	\end{prop}
	
	\begin{proof}
		We substitute  $\tr(e^{-tA^{\sharp}_{\chi}})$ 
		in the $\RHS$ of \eqref{resolvent}
		with the $\RHS$ of \eqref{selberg3}.
		Then, the $\RHS$ of \eqref{resolvent} is equal to 
		\begin{equation}\label{rhs}
			\RHS(\eqref{resolvent})=I(s_{1},s_{2};\chi)+H(s_{1},s_{2};\chi),
		\end{equation}
		where
		\begin{multline*}
			I(s_{1},s_{2};\chi) = \frac{1}{4\pi^2}\dim(V_{\chi})\Vol(X)
			\int_{0}^{\infty}\Big(e^{-t(s_{1}-1/2)^2}-e^{-t(s_{2}-1/2)^2}\Big)\\
			\times\bigg(\int_{\R}e^{-t\lambda^2}\lambda\pi\tanh\lambda\pi\,d\lambda\bigg)dt
		\end{multline*}
		and
		\begin{multline*}
			H(s_{1},s_{2};\chi) = \int_{0}^{\infty}\frac{1}{2\sqrt{4\pi t}}\Big(e^{-t(s_{1}-1/2)^2}-e^{-t(s_{2}-1/2)^2}\Big)\\
			\times\bigg(\sum_{[\gamma]\neq e} \frac{l(\gamma)\tr(\chi(\gamma))}{n_{\Gamma}(\gamma)\sinh(l(\gamma)/2)}
			e^{-\frac{l(\gamma)^{2}}{4t}}\bigg)dt.
		\end{multline*}
		We first compute $I(s_{1},s_{2};\chi)$.
		Substituting $\lambda'=\lambda\sqrt{t}$ and using $|\tanh(x)|\leq 1$, we find
		\begin{multline*}
			\int_{\R}\left|\bigg(e^{-t(s_{1}-1/2)^2}-e^{-t(s_{2}-1/2)^2}\bigg)e^{-t\lambda^2}\lambda\pi\tanh\lambda\pi\right|\,d\lambda\\
			\leq \left|\frac{e^{-t(s_{1}-1/2)^2}-e^{-t(s_{2}-1/2)^2}}{t}\right|\pi\int_0^\infty \lambda' e^{-(\lambda')^2}\,d\lambda',
		\end{multline*}
		where the latter integral over $\lambda'$ is finite and independent of $t$ and the factor in front is integrable over $t\in(0,\infty)$ by the Taylor expansion of $e^{-t(s_j-1/2)^{2}}$. Hence, we can interchange the integrals in the expression of $I(s_{1},s_{2};\chi)$ and compute the inner integral over $t$ to obtain
		\begin{multline}\label{id}
			I(s_{1},s_{2};\chi)=\frac{1}{4\pi^2}\dim(V_{\chi})\Vol(X)\\
			\times\int_{\R}\bigg(\frac{\lambda\pi\tanh\lambda\pi}{(s_{1}-\frac{1}{2})^{2}+\lambda^{2}}-\frac{\lambda\pi\tanh\lambda\pi}{(s_{2}-\frac{1}{2})^{2}+\lambda^{2}}\bigg) d\lambda.
		\end{multline}
		For computing $H(s_{1},s_{2};\chi)$, we use
		\begin{equation*}
			\int_{0}^{\infty} \frac{1}{\sqrt{4\pi t}}e^{-t(s-\frac{1}{2})^2}
			e^{-\frac{l(\gamma)^{2}}{4t}}\,dt=\frac{1}{2(s-\frac{1}{2})}
			e^{-(s-\frac{1}{2})l(\gamma)},
		\end{equation*}
		for $\RE(s-\frac{1}{2}),\RE((s-\frac{1}{2})^2)>0$ (\cite[(27), p. 146]{Er}).
		Hence,
		\begin{multline*}
			H(s_{1},s_{2};\chi)=
			\bigg(\frac{1}{2(s_{1}-\frac{1}{2})}\sum_{[\gamma]\neq e} \frac{l(\gamma)\tr(\chi(\gamma))}{2n_{\Gamma}(\gamma)\sinh(l(\gamma)/2)}
			e^{-(s_{1}-\frac{1}{2})l(\gamma)}\\
			-\frac{1}{2(s_{2}-\frac{1}{2})}\sum_{[\gamma]\neq e} \frac{l(\gamma)\tr(\chi(\gamma))}{2n_{\Gamma}(\gamma)\sinh(l(\gamma)/2)}
			e^{-(s_{2}-\frac{1}{2})l(\gamma)}\bigg).
		\end{multline*}
		By \eqref{logder}, we get
		\begin{equation}\label{hyp}
			H(s_{1},s_{2};\chi)=
			\frac{1}{2(s_{1}-\frac{1}{2})}L(s_{1};\chi)
			-\frac{1}{2(s_{2}-\frac{1}{2})}L(s_{2};\chi).
		\end{equation}
		The claim now follows by putting together \eqref{rhs}, \eqref{id} and \eqref{hyp}.
	\end{proof}
	
	\begin{prop}\label{merolog}
		The logarithmic derivative $L(s;\chi)$ of the Selberg zeta function $Z(s;\chi)$ extends to a meromorphic function in $s\in\C$ with singularities given by the following formal expression:
		\begin{equation}
			\sum_{j=0}^\infty\left[\frac{1}{s-\frac{1}{2}-i\mu_j}+\frac{1}{s-\frac{1}{2}+i\mu_j}\right] + \frac{\Vol(X)\dim(V_\chi)}{2\pi}\sum_{k=0}^\infty\frac{1+2k}{s+k}.\label{eq:merologFormalSingularities}
		\end{equation}
		where $(\lambda_j=\frac{1}{4}+\mu_j^2)_{j\in\Z^+}\subseteq\C$ are the eigenvalues of $\Delta_\chi^\sharp$ counted with algebraic multiplicity.
	\end{prop}
	
	\begin{proof}
		Consider the resolvent trace formula~\eqref{finalresolvent} for fixed $s_{2}\in\C$ with $\RE((s_2-\frac{1}{2})^2)\gg0$ and let $s=s_{1}\in\C$ vary. Multiplying both sides of \eqref{finalresolvent} with $2(s-\frac{1}{2})$ we see that $L(s;\chi)$ is meromorphic if the LHS of \eqref{finalresolvent} and $I(s,s_2;\chi)$ are meromorphic in $s\in\C$.
		
		The LHS of \eqref{finalresolvent}, that is the spectral side of the resolvent trace formula,
		is meromorphic since the resolvent $(A_\chi^\sharp+(s-1/2)^2)^{-1}$ is meromorphic, and the simple poles where $-(s-1/2)^2$ equals an eigenvalue of $A_\chi^\sharp$ give exactly the singularities of the first term in \eqref{eq:merologFormalSingularities}.
		
		In order to establish meromorphicity of $I(s,s_2;\chi)$ and obtain the singularities of the second term in \eqref{eq:merologFormalSingularities}, we recall the expression \eqref{id} and use the following identity (see e.g. \cite[p. 401]{Knapp}):
		\begin{equation*}
			\pi\lambda\tanh\pi\lambda = \sum_{\substack{n=-\infty\\ n\text{ odd}}}^{+\infty} \frac{\lambda^2}{\big(\frac{n}{2}\big)^{2}+\lambda^2}.
		\end{equation*}
		Then, we get
		\begin{multline}\label{idmer}
			2\Big(s-\frac{1}{2}\Big)I(s,s_{2};\chi)=\bigg(\frac{1}{4\pi^2}\dim(V_{\chi})\Vol(X)\bigg)\cdot\bigg(2(s-\frac{1}{2})\bigg)\\
			\times\bigg[\int_{\R}\bigg(\frac{1}{(s-\frac{1}{2})^{2}+\lambda^{2}}
			-\frac{1}{(s_{2}-\frac{1}{2})^{2}+\lambda^{2}}\bigg)
			\bigg(\sum_{\substack{n=-\infty\\ n\text{ odd}}}^{+\infty} \frac{\lambda^2}
			{\big(\frac{n}{2}\big)^{2}+\lambda^2}\bigg) d\lambda\bigg].
		\end{multline}
		Since one can interchange summation and integration on the 
		$\RHS$ of the equation above, we obtain the following expression
		\begin{equation*}
			I:=\sum_{\substack{n=-\infty\\n\text{ odd}}}^{+\infty}\Bigg( \int_{\R}\frac{\lambda^2}
			{\big((s-\frac{1}{2})^{2}+\lambda^{2}\big)\big((\frac{n}{2})^{2}+\lambda^2\big)}-\frac{\lambda^2} {\big((s_2-\frac{1}{2})^{2}+\lambda^{2}\big)\big((\frac{n}{2})^{2}+\lambda^2\big)} d\lambda\Bigg).
		\end{equation*}
		To compute the integral, note that for $a^2\neq b^2$ we can write
		$$ \frac{\lambda^2}{(a^2+\lambda^2)(b^2+\lambda^2)} = \frac{1}{a^2-b^2}\left(\frac{a^2}{a^2+\lambda^2}-\frac{b^2}{b^2+\lambda^2}\right) $$
		and use the following integral formula,  which is easily derived using the residue theorem:
		$$ \int_\R \frac{1}{a^2+\lambda^2}d\lambda = \frac{\pi}{a} \qquad \mbox{for }\RE (a)>0. $$
		This yields
		$$ I = \sum_{\substack{n=-\infty\\n\text{ odd}}}^{+\infty}\Bigg(\frac{\pi}{s+\frac{|n|-1}{2}}-\frac{\pi}{s_2+\frac{|n|-1}{2}}\Bigg) = \sum_{\substack{n=-\infty\\n\text{ odd}}}^{+\infty}\frac{\pi(s_2-s)}{(s+\frac{|n|-1}{2})(s_2+\frac{|n|-1}{2})}, $$
		and hence, \eqref{idmer} gives
		\begin{equation}\label{idmer2}
			2\Big(s-\frac{1}{2}\Big)I(s;s_{2},\chi)= \frac{1}{2\pi}\dim(V_{\chi})\Vol(X)\sum_{\substack{n=-\infty\\ n\text{ odd}}}^{+\infty} \frac{(s_2-s)(s-\frac{1}{2})}{(s+\frac{|n|-1}{2})(s_2+\frac{|n|-1}{2})}.
		\end{equation}
		The result now follows by substituting $k=\frac{|n|-1}{2}\in\Z^+$.
	\end{proof}
	
	\begin{thm}\label{selbergmero}
		Let $X=\Gamma\backslash \H^2$ be a compact hyperbolic surface and let $\chi\colon \Gamma\rightarrow \GL(V_{\chi})$ be a finite-dimensional, complex representation of $\Gamma$. Then, the twisted Selberg zeta function $Z(s;\chi)$ admits a holomorphic continuation to $\C$ with zeros given by the following formal product:
		$$ \prod_{j=0}^\infty\left(s-\frac{1}{2}-i\mu_j\right)\left(s-\frac{1}{2}+i\mu_j\right)\prod_{k=0}^\infty(s+k)^{(2g-2)\dim(V_\chi)(1+2k)}. $$
	\end{thm}
	
	\begin{proof}
		The theorem follows from Proposition~\ref{merolog} by integration and exponentiation. We only need to prove that the residues of $L(s;\chi)$ are positive integers. This follows from \eqref{eq:merologFormalSingularities} and the Gauss--Bonnet formula 
		\begin{equation*}
			\frac{\Vol(X)}{2\pi}= 2g-2,
		\end{equation*}
		where $g\geq2$ is the genus of the surface.
	\end{proof}
	
	\begin{rmrk}
		Since the twisted Bochner--Laplace operator $\Delta_\chi^\sharp$ is in general not self-adjoint, the eigenvalues $\lambda_j=\frac{1}{4}+\mu_j^2$ need not be real and $>0$ and hence, some of the zeros $\frac{1}{2}\pm i\mu_j$ and $-k$ of $Z(s;\chi)$ could coincide. Therefore, it is not possible in general to deduce more precise information about the degree of the zeros of $Z(s;\chi)$ from Theorem~\ref{selbergmero}.
	\end{rmrk}
	
	\begin{coro}\label{cor:RuelleMeromorphic}
		The twisted Ruelle zeta function $R(s;\chi)$ admits a meromorphic continuation to $\C$ with zeros and singularities given by the following formal product:
		$$ \prod_{j=0}^\infty\frac{(s-\frac{1}{2}-i\mu_j)(s-\frac{1}{2}+i\mu_j)}{(s+\frac{1}{2}-i\mu_j)(s+\frac{1}{2}+i\mu_j)}\times s^{(2g-2)\dim(V_\chi)}\prod_{k=1}^\infty(s+k)^{2(2g-2)\dim(V_\chi)}. $$
		In particular, $R(s;\chi)$ has a zero of order $(2g-2)\dim(V_\chi)$ at $s=0$.
	\end{coro}
	\begin{proof}
		The statement follows from Theorem \ref{selbergmero} and Lemma \ref{ruelleselberg}.
	\end{proof}
	
	\begin{thm}\label{selbergfe}
		The twisted Selberg zeta function satisfies the following functional equation.
		\begin{equation}\label{fefinal}
			\eta(s;\chi) = \frac{Z(s;\chi)}{Z(1-s;\chi)} = \exp\bigg[\dim(V_{\chi})\Vol(X)
			\int_0^{s-\frac{1}{2}}r\tan\pi r\,dr\bigg],
		\end{equation}
		where the integral is a complex line integral along any curve from $0$ to $s-\frac{1}{2}$.
	\end{thm}
	
	We remark that the value of the integral in \eqref{fefinal} depends on the chosen curve from $0$ to $s-\frac{1}{2}$, but the RHS of \eqref{fefinal} does not, because the residues of $r\tan\pi r$ are integer multiples of $\frac{1}{2\pi}$ and $\dim(V_\chi)\Vol(X)=2\pi(2g-2)\dim(V_\chi)$.
	
	\begin{proof}
		Write again $s=s_1$ and consider the transform $s\mapsto 1-s$ in \eqref{finalresolvent}. Then, the LHS of \eqref{finalresolvent} remains invariant under this transform.
		The $\RHS$ of \eqref{finalresolvent} is as in \eqref{rhs}, where
		$I(s,s_{2};\chi)$ is now given by \eqref{idmer2} and $H(s,s_{2};\chi)$ is given by \eqref{hyp}. (Note that we cannot use the identity \eqref{id} for $I(s,s_2;\chi)$ here since it is only valid for $\RE(s-\frac{1}{2})\gg0$ and hence only holds for $s$ and not for $1-s$.)
		Thus, if we consider the transform $s\mapsto 1-s$, then the $\RHS$ of \eqref{finalresolvent} will give 
		\begin{align}\label{fid}
			I(1-s;s_{2},\chi)=\frac{1}{4\pi}\dim(V_{\chi})\Vol(X)\sum_{\substack{n=-\infty\\ n\text{ odd}}}^{+\infty} \frac{s+s_2-1}{(1-s+\frac{|n|-1}{2})(s_2+\frac{|n|-1}{2})}
		\end{align}
		and
		\begin{equation}\label{fh}
			H(1-s;s_{2},\chi)=-\frac{1}{2(s-\frac{1}{2})}L(1-s;\chi)
			-\frac{1}{2(s_{2}-\frac{1}{2})}L(s_{2};\chi).
		\end{equation}
		Hence, by the above observation, \eqref{fid}
		and \eqref{fh}, we have
		\begin{equation*}
			H(s;s_{2},\chi)-H(1-s;s_{2},\chi)=-I(s;s_{2},\chi)+I(1-s;s_{2},\chi),
		\end{equation*}
		which further gives
		\begin{align*}
			\frac{L(s;\chi)+L(1-s;\chi)}{2(s-\frac{1}{2})} &= \frac{1}{4\pi}\dim(V_{\chi})\Vol(X)
			\sum_{\substack{n=-\infty\\ n\text{ odd}}}^{+\infty} \frac{2s-1}{(s+\frac{|n|-1}{2})(1-s+\frac{|n|-1}{2})}\\
			&= \frac{1}{2\pi}\dim(V_{\chi})\Vol(X)\sum_{k=0}^\infty \frac{2s-1}{(s+k)(1-s+k)}\\
			&= \frac{1}{2\pi}\dim(V_{\chi})\Vol(X)\sum_{k=0}^\infty\left(\frac{1}{1-s+k}-\frac{1}{s+k}\right).
		\end{align*}
		Using the classical identity
		$$ \pi\tan\pi x = 2x\sum_{k=0}^\infty\frac{1}{(k+\frac{1}{2})^2-x^2} = \sum_{k=0}^\infty\left(\frac{1}{k+\frac{1}{2}-x}-\frac{1}{k+\frac{1}{2}+x}\right) $$
		this can be rewritten as
		\begin{equation}\label{fe}
			L(s;\chi)+L(1-s;\chi)
			=\dim(V_{\chi})\Vol(X)\left(s-\frac{1}{2}\right)\tan\pi\left(s-\frac{1}{2}\right).
		\end{equation}
		We integrate both sides of \eqref{fe} over $s$ and exponentiate the result to get \eqref{fefinal}.
	\end{proof}
	
	\begin{coro}
		The twisted Ruelle zeta function satisfies the following functional equation
		\begin{equation}\label{rufe2}
			R(s;\chi)R(-s;\chi)=(2\sin\pi s)^{2(2g-2)\dim V_\chi}.
		\end{equation}
	\end{coro}
	
	\begin{proof}
		By Lemma \ref{ruelleselberg}, we have
		\begin{equation}\label{rufe1}
			R(s;\chi)R(-s;\chi)= \frac{Z(s;\chi)}{Z(1+s;\chi)}\frac{Z(-s;\chi)}{Z(1-s;\chi)}.
		\end{equation}
		Applying \eqref{fefinal} for $s$ and $-s$ shows:
		\begin{equation*}
			R(s;\chi)R(-s;\chi)=\exp\bigg[-\dim(V_{\chi})\Vol(X)\int_{s-\frac{1}{2}}^{s+\frac{1}{2}}r\tan\pi r\,dr\bigg],
		\end{equation*}
		where the integral is a complex line integral along any curve from 0 to $s-\frac{1}{2}$. To compute the integral, note that
		\begin{align*}
			\frac{d}{ds}\int_{s-\frac{1}{2}}^{s+\frac{1}{2}}r\tan\pi r\,dr &= \left(s+\frac{1}{2}\right)\tan\pi\left(s+\frac{1}{2}\right)-\left(s-\frac{1}{2}\right)\tan\pi\left(s-\frac{1}{2}\right)\\
			&= -\left(s+\frac{1}{2}\right)\cot\pi s+\left(s-\frac{1}{2}\right)\cot\pi s\\
			&= -\cot\pi s = -\frac{1}{\pi}\frac{d}{ds}\ln\sin\pi s.
		\end{align*}
		This implies
		\begin{equation}\label{eq:FunctEqRuelleUpToConstant}
			\exp\left[-2\pi\int_{s-\frac{1}{2}}^{s+\frac{1}{2}}r\tan\pi r\,dr\right] = c\cdot\sin^2\pi s
		\end{equation}
		for some constant $c\in\C$. To find $c$, we set $s=\epsilon$ in \eqref{eq:FunctEqRuelleUpToConstant} with $\epsilon>0$ and let $\epsilon\to0$. Then we can split the integral in \eqref{eq:FunctEqRuelleUpToConstant} into integrals over $(-\frac{1}{2}+\epsilon,0)$, $(0,\frac{1}{2}-\epsilon)$ and $(\frac{1}{2}-\epsilon,\frac{1}{2}+\epsilon)$. The integral over $(\frac{1}{2}-\epsilon,\frac{1}{2}+\epsilon)$ can be computed along a half circle in the upper half plane, giving the asymptotics
		$$ \int_{\frac{1}{2}-\epsilon}^{\frac{1}{2}+\epsilon}r\tan\pi r\,dr = -\pi i\cdot\Res_{r=\frac{1}{2}}r\tan\pi r+\mathcal{O}(\epsilon) = i+\mathcal{O}(\epsilon). $$
		The integrals over $(-\frac{1}{2}+\epsilon,0)$ and $(0,\frac{1}{2}-\epsilon)$ are equal, and we define
		$$ A(\epsilon) = -2\pi\int_0^{\frac{1}{2}-\epsilon}r\tan\pi r\,dr = -2\pi\int_0^{-\frac{1}{2}+\epsilon}r\tan\pi r\,dr. $$
		Then		
		\begin{align}\label{ruellezero}
			A(\epsilon)\notag&=-2\pi\int_{0}^{\frac{1}{2}-\epsilon} r \frac{\sin(2\pi r)}{1+\cos(2\pi r)}\,dr\\
			&= \left[r \log(1+\cos(2\pi r))\right]_{r=0}^{r=\frac{1}{2}-\epsilon}-\int_{0}^{\frac{1}{2}-\epsilon}\log(1+\cos(2\pi r))\,dr.
		\end{align}
		By \cite[p.153]{fried1986fuchsian}, the first term in \eqref{ruellezero}
		is given by 
		\begin{align*}
			\left[r \log(1+\cos(2\pi r))\right]_{r=0}^{r=\frac{1}{2}-\epsilon}
			&=\left(\frac{1}{2}-\epsilon\right)\log\left(1+\cos\left(2\pi\left(\frac{1}{2}-\epsilon\right)\right)\right)\\
			&=\left(\frac{1}{2}-\epsilon\right) \log\left(2\pi^{2} \epsilon^{2}\right)(1+O(\epsilon)\big).
		\end{align*}
		Moreover, by \cite[p.153]{fried1986fuchsian}, the latter integral in \eqref{ruellezero} is of 
		order $\frac{1}{2}\log \frac{1}{2}+O(\epsilon)$.
		Hence, as $\epsilon\rightarrow 0$
		\begin{equation}\label{asy}
			\exp(A(\epsilon))\sim 2\pi \epsilon.
		\end{equation}
		In total, we find that
		$$ \exp\left[-2\pi\int_{s-\frac{1}{2}}^{s+\frac{1}{2}}r\tan\pi r\,dr\right] = \exp\left[2A(\epsilon)-2\pi i+\mathcal{O}(\epsilon)\right] \sim (2\pi\epsilon)^2. $$
		This implies that $c=4$ and the proof is complete.
	\end{proof}
	
	The functional equation for $R(s;\chi)$ immediately gives the order of vanishing of $R(s;\chi)$ at $s=0$:
	
	\begin{coro}\label{cor:RuelleAtZero}
		The behavior of the twisted Ruelle zeta function $R(s;\chi)$ near $s=0$
		is given by
		\begin{equation}\label{eq:VanishingOrderRuelle}
			R(s;\chi)=\pm(2\pi s)^{\dim(V_{\chi})(2g-2)}+\text{higher order terms}.
		\end{equation}
	\end{coro}
	
	\begin{rmrk}\label{rem:SignInRuelleAtZero}
		The sign in \eqref{eq:VanishingOrderRuelle} can be determined explicitly. For this write
		$$ R(s;\chi) = \frac{Z(s;\chi)}{Z(s+1;\chi)} = \frac{Z(s;\chi)}{\eta(s+1;\chi)Z(-s;\chi)} $$
		and note that by \eqref{asy}:
		$$ \eta(s+1;\chi)^{-1} = \exp\left[-\dim(V_\chi)\Vol(X)\int_0^{s+\frac{1}{2}}r\tan\pi r\,dr\right] \sim (-2\pi s)^{\dim(V_\chi)(2g-2)} $$
		as $s\to0$. Moreover, by Theorem~\ref{selbergmero} we have
		$$ Z(s;\chi) \sim s^{m+\dim(V_\chi)(2g-2)} \qquad \mbox{as }s\to0, $$
		where $m$ denotes the multiplicity of the eigenvalue $\lambda=0$ of $\Delta_\chi^\sharp$. Hence,
		$$ R(s;\chi) \sim (-1)^m(2\pi s)^{\dim(V_\chi)(2g-2)}+\text{higher order terms}. $$
	\end{rmrk}
	
	\addcontentsline{toc}{section}{References}
	

\begin{thebibliography}{10}
	
	
	

	\bibitem{ANKER1991331}
	Jean-Philippe Anker, \emph{The spherical {F}ourier transform of rapidly
		decreasing functions. {A} simple proof of a characterization due to
		{H}arish-{C}handra, {H}elgason, {T}rombi, and {V}aradarajan}, J. Funct. Anal.
	\textbf{96} (1991), no.~2, 331--349.
	
	\bibitem{BM}
	Dan Barbasch and Henri Moscovici, \emph{{$L^{2}$}-index and the {S}elberg trace
		formula}, J. Funct. Anal. \textbf{53} (1983), no.~2, 151--201.
		
	\bibitem{BFS21}
	L\'{e}o B\'{e}nard, Jan Frahm and Polyxeni Spilioti, \emph{The twisted Ruelle
		zeta function on compact hyperbolic orbisurfaces and Reidemeister--Turaev
		torsion}, preprint, available at
		\href{https://arxiv.org/abs/2110.06683}{arXiv:2110.06683} (2021).

		
	\bibitem{BGV04}
	Nicole Berline, Ezra Getzler and Mich{\`e}le Vergne, \emph{Heat kernels and
		Dirac operators}, Grundlehren Text Editions, Springer-Verlag, Berlin, 2004.
	
	\bibitem{BO}
	Ulrich Bunke and Martin Olbrich, \emph{Selberg zeta and theta functions},
	Mathematical Research, vol.~83, Akademie-Verlag, Berlin, 1995, A differential
	operator approach.
	
	\bibitem{BO99}
	\bysame, \emph{Group cohomology and the singularities
	of the {S}elberg zeta function associated to a Kleinian group}, Ann. of Math.
	(2) \textbf{149} (1999), no.~2, 627--689.
	
	\bibitem{CDKP20}
	Mihajlo Ceki\'{c}, Semyon Dyatlov, Benjamin K\"{u}ster and Gabriel P. Paternain,
	\emph{The Ruelle zeta function at zero for nearly hyperbolic $3$-manifolds},
	Invent. Math.~\textbf{229} (2022), no.~1, 303--394.
	
	\bibitem{CD19}
	Yann Chaubet and Nguyen Viet Dang,
	\emph{Dynamical torsion for contact Anosov flows},
	preprint, available at
	\href{https://arxiv.org/abs/1911.09931}{arXiv:1911.09931} (2019).
	
	\bibitem{dang2019fried}
	Nguyen~Viet Dang, Colin Guillarmou, Gabriel Rivi\`ere, and Shu Shen, \emph{The
		{F}ried conjecture in small dimensions}, Invent. Math. \textbf{220} (2020),
	no.~2, 525--579.
	
	\bibitem{dang2020poincar}
	Nguyen~Viet Dang and Gabriel Rivi{\`e}re, \emph{Poincar\'{e} series and linking
		of Legendrian knots}, preprint, available at
	\href{https://arxiv.org/abs/2005.13235}{arXiv:2005.13235} (2020).
	
	\bibitem{DR}
	Georges {de Rham}, \emph{{Sur les nouveaux invariants topologiques de M.
			Reidemeister.}}, {Rec. Math. Moscou, n. Ser.} \textbf{1} (1936), 737--742.
	
	
	
	\bibitem{dyatlov2017ruelle}
	Semyon Dyatlov and Maciej Zworski, \emph{Ruelle zeta function at zero for
		surfaces}, Invent. Math. \textbf{210} (2017), no.~1, 211--229.
	
	
	\bibitem{Er}
	A.~Erd{\'e}lyi, W.~Magnus, F.~Oberhettinger, and F.~G. Tricomi, \emph{Tables of
		integral transforms. {V}ol. {I}}, McGraw-Hill Book Company, Inc., New
	York-Toronto-London, 1954.
	
	\bibitem{fedosova2020meromorphic}
	Ksenia Fedosova and Anke Pohl, \emph{Meromorphic continuation of {S}elberg zeta
		functions with twists having non-expanding cusp monodromy}, Selecta Math.
	(N.S.) \textbf{26} (2020), no.~1, Paper No. 9, 55.
	
	\bibitem{franz1935torsion}
	Wolfgang Franz, \emph{\"{U}ber die {T}orsion einer \"{U}berdeckung}, J. Reine
	Angew. Math. \textbf{173} (1935), 245--254.
	
	\bibitem{Fried}
	David Fried, \emph{Analytic torsion and closed geodesics on hyperbolic
		manifolds}, Invent. Math. \textbf{84} (1986), no.~3, 523--540.
	
	\bibitem{fried1986fuchsian}
	\bysame, \emph{Fuchsian groups and {R}eidemeister torsion}, The {S}elberg trace
	formula and related topics ({B}runswick, {M}aine, 1984), Contemp. Math.,
	vol.~53, Amer. Math. Soc., Providence, RI, 1986, pp.~141--163.
	
	
	\bibitem{gangolli1977}
	Ramesh Gangolli, \emph{The length spectra of some compact manifolds of negative
		curvature}, J. Differential Geom. \textbf{12} (1977), no.~3, 403--424.
	
	
	\bibitem{HC}
	Harish-Chandra, \emph{Harmonic analysis on real reductive groups. {III}. {T}he
		{M}aass-{S}elberg relations and the {P}lancherel formula}, Ann. of Math. (2)
	\textbf{104} (1976), no.~1, 117--201.
	
	\bibitem{Hej76}
	Dennis A. Hejhal, \emph{The Selberg trace formula for $\operatorname{PSL}(2,\mathbf{R})$.
	Vol. 1}, Lecture Notes in Mathematics, vol. 548, Springer-Verlag, Berlin, 1976.
	
	\bibitem{Hel}
	Sigurdur Helgason, \emph{Groups and geometric analysis}, Pure and Applied
	Mathematics, vol. 113, Academic Press, Inc., Orlando, FL, 1984.
	
	
	\bibitem{Knapp}
	Anthony~W. Knapp, \emph{Representation theory of semisimple groups}, Princeton
	Mathematical Series, vol.~36, Princeton University Press, Princeton, NJ,
	1986.
	
	\bibitem{MS91}
	Henri Moscovici and Robert~J. Stanton, \emph{R-torsion and zeta functions for
		locally symmetric manifolds}, Invent. Math. \textbf{105} (1991), no.~1,
	185--216.
	
	\bibitem{MSF}
	Henri Moscovici, {Robert J.} Stanton, and Jan Frahm, \emph{Holomorphic torsion
		with coefficients and geometric zeta functions for certain {H}ermitian
		locally symmetric manifolds}, preprint, available at
	\href{https://arxiv.org/abs/1802.08886}{arXiv:1802.08886} (2018).
	
	
	\bibitem{M1}
	Werner M\"{u}ller, \emph{A {S}elberg trace formula for non-unitary twists},
	Int. Math. Res. Not. IMRN \textbf{2011} (2011), no.~9, 2068--2109.
	
	\bibitem{mueller2020fried}
	\bysame, \emph{Ruelle zeta functions of hyperbolic manifolds and Reidemeister
		torsion}, J. Geom. Anal. \textbf{31} (2021), no.~12, 12501--12524.
	
	\bibitem{reidemeister1935homotopieringe}
	Kurt Reidemeister, \emph{Homotopieringe und {L}insenr\"{a}ume}, Abh. Math. Sem.
	Univ. Hamburg \textbf{11} (1935), no.~1, 102--109.
	
	\bibitem{Rue76}
	David Ruelle, \emph{Zeta-functions for expanding maps and Anosov flows},
	Invent. Math. \textbf{34} (1976), no.~3, 231--242.
	
	
	\bibitem{sel}
	Atle Selberg, \emph{Harmonic analysis and discontinuous groups in weakly
	symmetric Riemannian spaces with applications to Dirichlet series}, J. Indian
	Math. Soc. (N.S.) \textbf{20} (1956), 47--87.
	
	\bibitem{Shen2018}
	Shu Shen, \emph{Analytic torsion, dynamical zeta functions, and the {F}ried
		conjecture}, Analysis and PDE \textbf{11} (2018), no.~1.
	
	\bibitem{shen2020complex}
	\bysame, \emph{Complex Valued Analytic Torsion and Dynamical Zeta Function on
		Locally Symmetric Spaces}, Int. Math. Res. Not. IMRN (2021),
		\href{https://doi.org/10.1093/imrn/rnab335}{DOI:10.1093/imrn/rnab335}.
	
	\bibitem{shen2021analytic}
	\bysame, \emph{Analytic torsion, dynamical zeta function, and the {F}ried
		conjecture for admissible twists}, Communications in Mathematical Physics
	(2021), 1--41.
	
	\bibitem{Sh}
	Michail~A. Shubin, \emph{Pseudodifferential operators and spectral theory},
	Springer Series in Soviet Mathematics, Springer-Verlag, Berlin, 1987.
	
	\bibitem{SB}
	Barry Simon, \emph{Trace ideals and their applications}, second ed.,
	Mathematical Surveys and Monographs, vol. 120, American Mathematical Society,
	Providence, RI, 2005.
	
	\bibitem{Spilioti2018}
	Polyxeni Spilioti, \emph{Selberg and {R}uelle zeta functions for non-unitary
		twists}, Ann. Global Anal. Geom. \textbf{53} (2018), no.~2, 151--203.
	
	\bibitem{spilioti2020functional}
	\bysame, \emph{Functional equations of {S}elberg and {R}uelle zeta functions
		for non-unitary twists}, Ann. Global Anal. Geom. \textbf{58} (2020), no.~1,
	35--77.
	
	\bibitem{spilioti2020twisted}
	\bysame, \emph{Twisted {R}uelle zeta function and complex-valued analytic
		torsion}, preprint, available at
		\href{https://arxiv.org/abs/2004.13474}{arXiv:2004.13474} (2020).
	
	\bibitem{Wab}
	Nolan~R. Wallach, \emph{Harmonic analysis on homogeneous spaces}, Marcel
		Dekker, Inc., New York, 1973, Pure and Applied Mathematics, No. 19.
	
	\bibitem{Wa}
	\bysame, \emph{On the {S}elberg trace formula in the case of compact quotient},
		Bull. Amer. Math. Soc. \textbf{82} (1976), no.~2, 171--195.
	
	\bibitem{Wo}
	Artur~Wotzke, \emph{{Die Ruellesche Zetafunktion und die analytische Torsion
		hyperbolischer Mannigfaltigkeiten}}, {PhD} thesis, Bonn, Bonner Mathematische
		Schriften, 2008.
	
\end{thebibliography}

\providecommand{\bysame}{\leavevmode\hbox to3em{\hrulefill}\thinspace}
\providecommand{\MR}{\relax\ifhmode\unskip\space\fi MR }
\providecommand{\MRhref}[2]{%
	\href{http://www.ams.org/mathscinet-getitem?mr=#1}{#2}
}
\providecommand{\href}[2]{#2}

	\contact
	
\end{document}